\documentclass{amsart}
\usepackage{tikz, amsmath, amscd, amssymb, amsfonts, stmaryrd, vmargin, amsthm, soul, bm, graphicx, float, exscale, relsize, etoolbox, mathrsfs, url}
\usepackage[utf8]{inputenc}

\theoremstyle{plain}        \newtheorem{thm}{Theorem}

\theoremstyle{plain}        \newtheorem{pro}[thm]{Proposition}

\theoremstyle{plain}        \newtheorem{lem}[thm]{Lemma}

\theoremstyle{plain}        

\theoremstyle{plain}        \newtheorem{rem}[thm]{Remark}

\theoremstyle{plain}        

\theoremstyle{plain}        

\numberwithin{equation}{section}
\numberwithin{thm}{section}

\usepackage{comment}

\begin{document}

\title{On compactness of the $\bar{\partial}$-Neumann operator on Hartogs domains}

\author{Muzhi Jin}
\date{}

\maketitle
\begin{abstract}
\sloppy We show that Property $(P)$ of $\partial\Omega$, compactness of the $\bar{\partial}$-Neumann operators $N_1$, and compactness of Hankel operator on a smooth bounded pseudoconvex Hartogs domain $\Omega={\{(z, w_1, w_2,\dots, w_n)\in \mathbb{C}^{n+1} \mid\sum_{k=1}^{n} |w_k|^2<e^{-2\varphi(z)}, z\in D\}}$ are equivalent, where $D$ is a smooth bounded connected open set in $\mathbb{C}$. 
\end{abstract}

\section{Introduction}

Let $\Omega$ be a bounded pseudoconvex domain in $\mathbb{C}^n$, $\mathit{L}_{(0,q)}^{2}(\Omega)$ be the space of $(0,q)$-forms with coefficients in $L^{2}(\Omega)$. The complex Laplacian operator $\Box_q=\bar{\partial}\bar{\partial}^{*}+\bar{\partial}^{*}\bar{\partial}$ is a densely defined, closed, self-adjoint operator on $L_{(0,q)}^{2}(\Omega)$. By H\"{o}rmander's $L^2$ estimates \cite{H}, $\Box_q$ has a bounded inverse, called the $\bar{\partial}$-Neumann operator $N_q$ and $\bar\partial^* N_q$ provides Kohn's canonical solution operator of $\bar\partial$ equation. The regularity of the canonical solution to $\bar\partial$ equation is one of the most fundamental problems in several complex variables and partial differential equations. The deep relation between the compactness of $N_q$ and the global regularity of canonical solution of $\bar\partial$ follows from the result of Kohn and Nirenberg \cite{KN} that if $N_q$ is compact in $L^2_{(0,q)}(\Omega)$, then $N_q$ is compact in the $L^2$-Sobolev spaces, thus the global regularity of the canonical solution holds. For the deep theory of regularity of $\bar{\partial}$-Neumann problem, the interested readers may refer to \cite{BS}.

 Catlin introduced a concept of Property $(P_q)$: 
 A compact set $K\in\mathbb{C}^{n}$ is said to satisfy Property $(P_q)$ if for every positive number $M$, there exists a neighborhood $U$ of $K$ and a $C^2$ function $\lambda$ on $U$, $0 \leq \lambda \leq 1$, such that for all $z\in K$, the sum of any $q$ eigenvalues of the Hermitian form $(\frac{\partial^{2}\lambda}{\partial z_j\partial\bar{z}_k}(z))_{jk}$ is at least $M$. 
Moreover Catlin proved the following fundamental result (see Theorem 1 in \cite{Ca}, \cite{Str1} as well) which characterized the compactness of $N_q$ in $L^2$. 

\begin{thm}\label{Catlin;thm1}
Let $\Omega$ be a bounded pseudoconvex domain in $\mathbb{C}^{n}$. Let $1\leq q\leq n$. If $\partial\Omega$ satisfies Property $(P_q)$, then $N_q$ is compact.
\end{thm}


Along this line of the study, the equivalence between Property $(P_q)$ of the boundary the domain and the compactness of the $\bar{\partial}$-Neumann operator $N_q$ has been established in some pseudoconvex domains. Fu and Straube proved that Property $(P_q)$ of the boundary and compactness of $N_q$ are equivalent on locally convexitiable domains \cite{FS1}. 
Christ and Fu proved that on smoothly bounded pseudoconvex Hartogs domain $\Omega$ in $\mathbb{C}^2$, the boundary $\partial\Omega$ satisfies Property $(P)$ if and only if $N_1$ is compact \cite{CF}. Here and from now on, we use $(P)$ to denote $(P_1)$ as the main focus is the compactness of $N_1$. Our note is motivated by the theorems in \cite{FS1} \cite{CF} to study the compactness of $N_1$ and Property $(P)$, and we try to generalize the result of Christ and Fu to the higher dimension.

The proof of Christ and Fu involves the intricate study of the asymptotic behavior of the first eigenvalues of Schr\"{o}dinger operators with magnetic and non-magnetic fields. We introduce the terminology of the Schr\"{o}dinger operator (more details can be found in \cite{FS3}, \cite{CF}). 
Let $D$ be a bounded domain in $\mathbb{C}$ and $\varphi\in\mathit{C}^2(\overline{D})$. $S_{\varphi}=-[(\partial_x+i\varphi_y)^{2}+(\partial_y-i\varphi_x)^{2}]+\Delta\varphi$ is the magnetic Schr\"{o}dinger operator and $S_{\varphi}^{0}=-\Delta+\Delta\varphi$ is the corresponding nonmagnetic Schr\"{o}dinger operator. Let $\lambda_{\varphi}(D)$ and $\lambda_{\varphi}^{0}(D)$ be the smallest eigenvalue of $S_{\varphi}$, $S_{\varphi}^{0}$ on $D$ respectively. Let $L_{\varphi}=-e^{\varphi}(\partial/\partial z)(e^{-\varphi} \cdot)=-\partial_{z}+\varphi_{z}$ be the first-order differential operator defined $L^{2}(D)$ in the sense of distributions and let $\bar{L}_{\varphi}=e^{-\varphi}(\partial/\partial\bar{z})(e^{\varphi}\cdot)=\partial_{\bar{z}}+\varphi_{\bar{z}}$ be the formal adjoint of $\mathit{L}_{\varphi}$.  Note that $S_{\varphi}=4\bar{L}_{\varphi} L_{\varphi}$. Therefore,
\begin{equation}\label{mj;eq1.1}
\begin{split}
\lambda_{\varphi}(D)&=\inf\{4\int_{D}|\mathit{L}_{\varphi} u|^{2}\,/\int_{D}|u|^{2}\,; u\in\mathit{C}_{0}^{\infty}(D), u\not\equiv0\}\\
&=\inf\{4\int_{D}|u_z|^{2}e^{2\varphi}\,/\int_{D}|u|^{2}e^{2\varphi}\,; u\in\mathit{C}_{0}^{\infty}(D), u\not\equiv0\}.
\end{split}
\end{equation}
The key step of their proof is the following deep theorem regarding Schr\"odinger operators (see Theorem 1.5 in \cite{CF}), which also plays a fundamental role in the proof of our result.

\begin{thm}\label{Christ-Fu;thm1}
Let $\varphi$ be subharmonic such that $\Delta\varphi$ is H\"{o}lder continuous of some positive order. If $\sup_m \lambda_{m\varphi}^{0}<\infty$, then $\lim\inf_{m\rightarrow\infty} \lambda_{m\varphi}<\infty$.
\end{thm}

Let $A^{2} (\Omega)$ be the space of holomorphic $L^{2}$ functions on $\Omega$ and $\phi\in L^{\infty} (\Omega)$. The Hankel operator $H_{\phi}:A^{2} (\Omega)\rightarrow L^{2} (\Omega)$ with symbol $\phi$ is defined by
$$H_{\phi} f=[\textbf{B}, \phi]f=(\textbf{B}-I)(\phi f).$$
Here $I$ is the identity mapping, $\textbf{B}$ is the Bergman projection. Therefore, when $\phi\in C^{1}(\overline{\Omega})$, $H_{\phi} f=-\bar{\partial}^{*}N_1(f\bar{\partial}\phi)$. 

On bounded pseudoconvex domains in $\mathbb{C}^{n}$, compactness of $\bar{\partial}$-Neumann operator $N_1$ implies compactness of the Hankel operator (see Proposition 2.5). The inverse direction is still open in general. On a smooth bounded pseudoconvex Hartogs domain in $\mathbb{C}^{2}$, it is proved by \c{S}ahuto\u{g}lu and Zeytuncu that compactness of $\bar{\partial}$-Neumann operator $N_1$ and compactness of the Hankel operator are equivalent \cite{SZ}. We also discuss the compactness of Hankel operator and $N_1$ in this note. 

\medskip

Using the ideas developed in \cite{CF} \cite{FS3} \cite{SZ}, we  prove the following result in this note. As pointed out to us by Prof.\c{S}ahuto\u{g}lu, the main result can be generalized to the weighted $L^2$ version using the similar arguments. 

\begin{thm}\label{mj;thm1}
Let $D\subset \mathbb{C}$ be a smooth bounded domain, $\varphi$ be a subharmonic function on $D$. Let $\Omega=\{(z, w_1, w_2,\dots, w_n)\mid\sum_{k=1}^{n} |w_k|^2<e^{-2\varphi(z)}, z\in D\}$ be a smooth pseudoconvex Hartogs domain in $\mathbb{C}^{n+1}$. Then the following statements are equivalent:
\begin{itemize}
\item $\partial\Omega$ satisfies Property (P).
\item $\bar{\partial}$-Neumann operator $\mathit{N_1}$ is compact on $L^2_{(0, 1)}(\Omega)$.
\item The Hankel operator $H_{\phi}$ is compact on $A^2 (\Omega)$ for all $\phi\in C^{\infty}(\overline{\Omega})$.
\end{itemize}
\end{thm}

\section{Preliminaries}

Using the result of Sibony(\cite{Sib}, see \cite{F} as well), there are some equivalent descriptions of Property (P) in one-dimensional case 
(see Proposition 5 in \cite{FS3}). 

\begin{pro}\label{mj;pro1}
Let $\mathit{K}$ be a compact subset of $\mathbb{C}$. The following statements are equivalent:\\
(1) $\mathit{K}$ satisfies Property (P).\\
(2) $\mathit{K}$ has empty fine interior.\\
(3) $\mathit{K}$ supports no zero function in $\mathit{W}_{0}^{1}(\mathbb{C})$\\
(4) For any sequence of open sets $\{\mathit{U}_j\}_{j=1}^{\infty}$ such that $\mathit{K}\subset\subset\mathit{U}_{j+1}\subset\subset\mathit{U}_{j}$ and $\cap_{j=1}^{\infty}\overline{\mathit{U}_{j}}=\mathit{K}$, $\lambda(\mathit{U}_{j})\rightarrow\infty$ as $j\rightarrow\infty$. $\lambda(\mathit{U}_{j})$ denotes the first eigenvalue of the Dirichlet problem for the Laplacian on $\mathit{U}_{j}$, i.e. $\lambda(\mathit{U}_{j})=\inf\{\int |\nabla u|^2/\int |u|^2; u\in\mathit{C}_{0}^{\infty}(\mathit{U}_{j})\}$
\end{pro}

On a subset $U\subset \mathbb{C}$, $z\in U$ is called a fine interior point of $U$ if $\lim_{r\rightarrow 0}(\sigma(\partial B(z,r)\cap U)/2\pi r)=1$, where $\sigma$ denotes the length of arcs. Readers can find the proof in Proposition 4.17 in \cite{Str2}. 

In high-dimensional case, following Sibony's arugments of Property (P), some necessary ingredients of the proof is described as follows (more details can be found in \cite{FS2} or Proposition 4.10 on P.88 in \cite{Str2}). $U$ is an open set of $\mathbb{C}^{n}$ and $P(U)$ denotes the set of all continuous plurisubharmonic functions on $U$. $K$ is a compact subset of $\mathbb{C}^n$ and $P(K)$ denotes the closure in the algebra of continuous functions on $K$ that belong to $P(V)$ for an open neighborhood $V$ of $K$ ($V$ is allowed to depend on the function). 

\begin{pro}\label{mj;pro2}
Let $K$ be a compact subset of $\mathbb{C}^n$, then the following statements are equivalent.\\
(1) $K$ satisfies Property (P).\\
(2) $-|z|^2\in P(K)$.\\
(3) $P(X)=C(X).$
\end{pro}

To establish the connection from the compactness of Hankel operator to the compactness of $\bar{\partial}$-Neumann operator, we are using the following estimate proved in Lemma 3 of \cite{SZ}. For completeness, we include the proof here.

\begin{lem}\label{mj;lem3}
Let $\Omega$ be a bounded pseudoconvex domain in $C^{n}$ and $\phi\in C^{1}(\bar{\Omega})$. If the Hankel operator $H_{\phi}$ is compact on $A^{2}(\Omega)$, then for any $\epsilon>0$, there exists $C_\epsilon>0$, such that 
\begin{equation}\label{mj;eq2}
\|H_\phi h\|^{2}\leq\epsilon\|h\bar{\partial}\phi\|\|h\|+C_{\epsilon}\|h\bar{\partial}\phi\|_{-1}\|h\|
\end{equation}
for any $h\in A^{2}(\Omega)$.
\end{lem}
\begin{proof}

Since $\Omega$ is bounded and pseudoconvex, for any $u\in L^{2}(\Omega)$, 
\begin{equation*}
\begin{split}
\|\bar{\partial}N u\|^2+\|\bar{\partial}^{*}N u\|^2 =\langle u,N u\rangle\leq\|u\|\|N u\|\leq C\|u\|^{2},
\end{split}
\end{equation*}
where $C$ just depends on $\Omega$. When $h\in A^{2}(\Omega)$, 
\begin{equation*}
\|H_\phi h\|^{2}=\langle H_{\phi}^{*}H_{\phi} h,h\rangle\leq\|H_{\phi}^{*}H_{\phi} h\|\|h\|.
\end{equation*}

Compactness of Hankel operator is equivalent to compactness of its adjoint operator $H_{\phi}^{*}$. By compactness estimate of $H_{\phi}^{*}$, (see for example), for any $\epsilon>0$, there exists a compact operator $K_\epsilon$, so that
\begin{equation*}
\|H_{\phi}^{*}H_{\phi} h\|\leq\frac{\epsilon}{2C}\|H_{\phi} h\|+\|K_{\epsilon} H_{\phi} h\|=\frac{\epsilon}{2C}\|\bar{\partial}^{*}N(h\bar{\partial}\phi)\|+\|K_{\epsilon} \bar{\partial}^{*}N(h\bar{\partial}\phi)\|.
\end{equation*}
Also since $\bar{\partial}^{*}N$ is bounded, $K_{\epsilon}\bar{\partial}^{*}N$ is compact. And by Rellich's Lemma, for that $\epsilon$, there exists $C_{\epsilon}>0$, such that
\begin{equation*}
\|K_{\epsilon} \bar{\partial}^{*}N(h\bar{\partial}\phi)\|\leq\frac{\epsilon}{2}\|h\bar{\partial}\phi\|+C_{\epsilon}\|h\bar{\partial}\phi\|_{-1}.
\end{equation*}
Overall, the estimate ($\ref{mj;eq2}$) is achieved.

\end{proof}
\begin{rem}
The converse is also true and readers can find the proof in \cite{SZ}.
\end{rem}

It is well known that compactness of $N_q$ implies compactness of canonical solution operators $\bar{\partial}^{*}N_q$ and $\bar{\partial}^{*}N_{q+1}$. Moreover, we have the following conclusion(see Remark 1 in \cite{CeSa}).

\begin{pro}\label{mj;pro3}
Let $\Omega$ be a bounded pseudoconvex domain in $\mathbb{C}^{n}$. Then the Hankel operator $H_{\phi}$ is compact on $A^2 (\Omega)$ for all $\phi\in C^{\infty}(\overline{\Omega})$ when $N_1$ is compact on $L_{(0,1)}^{2} (\Omega)$. 
\end{pro}
\begin{proof}
For any $f\in C_{(0,1)}^{\infty}(\Omega)$,
\begin{equation*}
\begin{split}
N_1 &=N_1(\bar{\partial}^{*}\bar{\partial}+\bar{\partial}\bar{\partial}^{*})N_1=\bar{\partial}^{*}N_2(N_2\bar{\partial})+(N_1\bar{\partial})\bar{\partial}^{*}N_1\\
&=\bar{\partial}^{*}N_2(\bar{\partial}^{*}N_2)^{*}+(\bar{\partial}^{*}N_1)^{*}\bar{\partial}^{*}N_1 .
\end{split}
\end{equation*}

Since $C_{(0,1)}^{\infty}(\Omega)$ is dense in $L_{(0,1)}^{2}(\Omega)$, compactness of $N_1$ on $L_{(0,1)}^{2}(\Omega)$ implies compactness on $C_{(0,1)}^{\infty}(\Omega)$. Since both $\bar{\partial}^{*}N_2(\bar{\partial}^{*}N_2)^{*}$ and $(\bar{\partial}^{*}N_1)^{*}\bar{\partial}^{*}N_1$ are positive, these two operators are compact when $N_1$ is compact. Therefore $\bar{\partial}^{*}N_1$ and $\bar{\partial}^{*}N_2$ are compact on $C_{(0,1)}^{\infty}(\Omega)$ and $C_{(0,2)}^{\infty}(\Omega)$ respectively. Hence for any $\phi\in C^{\infty}(\overline{\Omega})$, $f\in A^{2}(\Omega)$, $H_{\phi} f=-\bar{\partial}^{*}N(f\bar{\partial}\phi)$ is compact.
\end{proof}

We discuss a class of Hartogs domains in $\mathbb{C}^{n+1}$ in this paper and introduce following lemmas.

\begin{lem}\label{mj;lem1}
Let $\Omega=\{(z, w_1, w_2,\dots, w_n)\mid\sum_{k=1}^{n} |w_k|^2<e^{-2\varphi(z)}, z\in D\}$ be a smooth pseudoconvex Hartogs domain in $\mathbb{C}^{n+1}$, where $D$ is a smooth bounded domain in $\mathbb{C}$. $\pi$ denotes the projection from $\partial\Omega$ to $\bar{D}$. Then for any compact subset $K$ in $D$, $\pi^{-1}(K\cap\{z\in D\mid\Delta\varphi(z)=0\})$ is the set of all weakly pseudoconex points in $\pi^{-1}(K)$. 
\end{lem}
\begin{proof}
The defining function $\rho=\sum_{k=1}^{N} |w_k|^2-e^{-2\varphi(z)}$. The hessian of $\rho$ is given by:
\begin{equation*}
\mathbb{H}(\rho)=
\begin{bmatrix}
-4e^{-2\varphi}|\frac{\partial \varphi}{\partial z}|^2+2e^{-2\varphi} \frac{\partial^{2}\varphi}{\partial z\partial\bar{z}} & 0 & \dots & 0\\
0 & 1 & \dots & 0\\
\vdots & \vdots & \ddots & 0\\
0 & 0 & \dots & 1
\end{bmatrix},
\end{equation*}
and the complex tangent space $\mathit{T}_{(z, w_1, w_2,..., w_n)}^{\mathbb{C}}(\partial\Omega)=\{(\tau,\xi_1,\dots,\xi_n)\mid 2e^{-2\varphi(z)} \frac{\partial\varphi}{\partial z} \tau+\overline{w_1}\xi_1+\dots+\overline{w_n}\xi_n=0\}$.

Since $\Omega$ is smooth, $\lim_{z\rightarrow\partial D}\varphi (z)=+\infty$, and thus $\partial\Omega\cap\{w_1=\dots =w_n=0\}=\{(z,0,\dots,0); z\in\partial D\}$. Also by the assumption $K\subset\subset D$, any point on $\pi^{-1}(K)$ has at least one nonzero $w_k$. Without loss of generality, let $w_1\neq0$. Then if $\frac{\partial\varphi}{\partial z}=0$, $\mathit{T}_{(z, w_1, w_2,..., w_n)}^{\mathbb{C}}(\partial\Omega)=\{(\tau, -\frac{\overline{w_2}\xi_2+\dots+\overline{w_n}\xi_n}{\overline{w_1}}, \xi_2,\dots,\xi_n)\}$. The Levi form is nonnegative, and by taking ${\xi_2=\dots=\xi_n=0}$, it achieves 0 when $\frac{\partial^{2}\varphi}{\partial z\partial\bar{z}}=\frac{1}{4}\Delta\varphi=0$. On the other hand, if $\frac{\partial\varphi}{\partial z}\neq0$, $\mathit{T}_{(z, w_1, w_2,..., w_n)}^{\mathbb{C}}(\partial\Omega)=\{(-\frac{\overline{w_1}\xi_1+\dots+\overline{w_n}\xi_n}{2e^{-2\varphi} \frac{\partial\varphi}{\partial z}},\xi_1,\dots,\xi_n)\}$. The Levi forms applying to these tangent vectors are:
\begin{equation}\label{mj,eq1}
\begin{split}
~ &-\frac{|\overline{w_1}\xi_1+\dots+\overline{w_n}\xi_n|^2}{e^{-2\varphi}}+\frac{1}{2}e^{2\varphi} \frac{\partial^{2}\varphi}{\partial z\partial\bar{z}} \frac{|\overline{w_1}\xi_1+\dots+\overline{w_n}\xi_n|^2}{ |\frac{\partial\varphi}{\partial z}|^2}+|\xi_1|^2+\dots+|\xi_n|^2\\
&=(-\frac{|\overline{w_1}\xi_1+\dots+\overline{w_n}\xi_n|^2}{|\overline{w_1}|^2+\dots+|\overline{w_n}|^2}+|\xi_1|^2+\dots+|\xi_n|^2)+\frac{1}{2}e^{2\varphi} \frac{\partial^{2}\varphi}{\partial z\partial\bar{z}} \frac{|\overline{w_1}\xi_1+\dots+\overline{w_n}\xi_n|^2}{ |\frac{\partial\varphi}{\partial z}|^2}.
\end{split}
\end{equation}
The first part of the last line of $\eqref{mj,eq1}$ is nonnegative and achieves 0 if and only if $\frac{\xi_k}{\overline{w_k}}=c$, for all $1\leq k\leq n$. The second part is nonnegative and getting zero when $\frac{\partial^{2}\varphi}{\partial z\partial\bar{z}}=\frac{1}{4}\Delta\varphi=0$. In general, the weakly pseudoconvex points of $\pi^{-1}(K)$ are excatly $\pi^{-1}(K\cap\{z\in D\mid\Delta\varphi(z)=0\})$.
\end{proof}

\begin{lem}\label{mj;lem2}
Let $\Omega$ be as above, $K$ be a compact set in $D$. If $K$ satisfies Property $(P)$, so does $\pi^{-1}(K)$.
\end{lem}
\begin{proof}
By proposition $\ref{mj;pro2}$ (2), it suffices to check whether the function $-|z|^2-|w_1|^2-\dots-|w_n|^2$ belongs to $P(\pi^{-1}(K))$. On $\pi^{-1}(K)$, $-|z|^2-|w_1|^2-\dots-|w_n|^2=-|z|^2-e^{-\varphi(z)}$. Where $-|z|^2-e^{-\varphi(z)}$ can be viewed as a function in $C(K)$. By assumption and (3) of Proposition $\ref{mj;pro2}$, $-|z|^2-e^{-\varphi(z)}$ can be approximated uniformly on $K$ by subharmonic funtions near $K$. Moreover, these functions can be viewed as a plurisubharmonic functions of $(z,w_1,\dots,w_n)$ near $\pi^{-1}(K)$. Thus $-|z|^2-|w_1|^2-\dots-|w_n|^2\in P(\pi^{-1}(K))$. Hence $\pi^{-1}(K)$ satifies Property (P).
\end{proof}

\begin{lem}\label{mj;lem4}
Let $\Omega$ be as above, $\Omega'=\{(z', w_1', w_2',\dots, w_n')\mid\sum_{k=1}^{n} |w_k'|^2<e^{-2\varphi(z)-2\alpha}, z\in D\}$, where $\alpha=\inf_{z\in D} \varphi(z)$. Let $T$: $\Omega\rightarrow\Omega'$ be given by $T(z, w_1, \cdots, w_n)=(z, \frac{w_1}{t}, \cdots, \frac{w_n}{t}$), where $t=\max\{e^{-\alpha},1\}$. Then Hankel operator $H_{\phi}$ is compact on $A^2 (\Omega)$ for all $\phi\in C^{\infty}(\overline{\Omega})$ if and only if the Hankel operator $H_{\phi'}$ is compact on $A^2 (\Omega')$ for all $\phi'\in C^{\infty}(\overline{\Omega'})$.

\end{lem}
\begin{proof}
Note that since $\varphi$ is subharmonic, $\alpha$ and $t$ are finite. $T$ is a biholomorphism and $T, T^{-1}$ are both smooth up to the boundary. Let $\{f_{j}'\}_{j=1}^{\infty}$ be a bounded sequence of functions in $A^{2}(\Omega')$, $\phi'\in C^{\infty}(\overline{\Omega'})$. By the definition of $T$, $\{f_{j}=f_{j}'\circ T\}_{j=1}^{\infty}$ is a bounded sequence functions in $A^{2}(\Omega)$ and $\phi=\phi'\circ T\in C^{\infty}(\bar{\Omega})$. Assume $H_{\phi}$ is compact, then there exists a subsequence of $\{f_{j_k}\}$ such that $\{H_{\phi} f_{j_k}\}$ is Cauchy in $L^2(\Omega)$. Therefore $\{H_{\phi} f_{j_k}\circ T^{-1}\}=\{H_{\phi'} f_{j_k}'\}$ is also Cauchy in $L^2(\Omega')$, so $H_{\phi'}$ is compact. Similarly, the other direction is also true.
\end{proof}

We also need following lemmas, and proofs follow approaches in Lemma 1 and Lemma 2 in \cite{SZ}. 
\begin{lem}\label{mj;lem5}
Let $B(0,r)=\{w\in\mathbb{C}; |w|<r\}$ for $0<r<\infty$ and $d(w)$ be the distance from $w$ to $\partial B(0,r)$. Then for any positive integer $n$,
\begin{equation}\label{mj;eq3}
\int_{B(0,r)} (d(w))^{2} |w|^{2n}d\,V(w)\leq \frac{r^2}{2n^2} \int_{B(0,r)} |w|^{2n}d\,V(w) .
\end{equation}
\end{lem}
\begin{proof}
For right hand side of the last inequality,
$$
\int_{B(0,r)} |w|^{2n}d\,V(w)=2\pi\int_{0}^{r} |w|^{2n+1}\,d\,|w|=\frac{\pi r^{2n+2}}{n+1}.
$$
On the other hand,
\begin{equation*}
\begin{split}
\int_{B(0,r)} (d(w))^{2} |w|^{2n}d\,V(w) &=2\pi\int_{0}^{r}(r-|w|)^{2} |w|^{2n+1}\,d\,|w|\\
&=\frac{\pi r^{2n+4}}{(n+1)(n+2)(2n+3)}.
\end{split}
\end{equation*}
Thus we have
\begin{equation*}
\frac{n^{2}\int_{A(0,r)} (d(w))^{2} |w|^{2n}d\,V(w)}{\int_{A(0,r)} |w|^{2n}d\,V(w)}=\frac{n^{2}r^{2}}{(n+2)(2n+3)}\leq \frac{r^2}{2}.
\end{equation*}
where $\frac{r^2}{2}$ is a finite number. Therefore $\eqref{mj;eq3}$ is satisfied.
\end{proof}

\begin{lem}\label{mj;lem6}
Let $\Omega$ be a bounded domain in $\mathbb{C}^{n}$. For any $f\in W^{-1}(\Omega)$, there exists $C>0$, so that $$\|f\|_{-1}\leq C \|d_{\Omega}f\|,$$
where  $d_{\Omega}(z)$ denoted the distance between $z$ and the boundary of $\Omega$.
\end{lem}
\begin{proof}
By the definition of $W^{-1}$ norm, 
\begin{equation*}
\begin{split}
\|f\|_{-1} &=\sup \{|<f,\phi>|;\phi\in C_{0}^{\infty}(\Omega),\|\phi\|_{1}\leq1\}\\
&\leq\|d_{\Omega}f\|\sup\{\|\phi/d_{\Omega}\|;\phi\in C_{0}^{\infty}(\Omega),\|\phi\|_{1}\leq1\}\\
&\leq C \|d_{\Omega}f\|.
\end{split}
\end{equation*}
Last inequality is an application of Hardy-Littlewood lemma (see the proof of Theorem C.3 on P.347 in \cite{ChSh}). 
\end{proof}

\section{Proof of Theorem $\ref{mj;thm1}$}

In this section, we prove the following theorem, following the approaches in \cite{CF}, \cite{FS3} and \cite{SZ}.

\begin{thm}
Let $D\subset \mathbb{C}$ be a smooth bounded domain, $\varphi$ be a subharmonic function on $D$. Let $\Omega=\{(z, w_1, w_2,\dots, w_n)\mid\sum_{k=1}^{n} |w_k|^2<e^{-2\varphi(z)}, z\in D\}$ be a smooth pseudoconvex Hartogs domain in $\mathbb{C}^{n+1}$. Then the following statements are equivalent:\\
(1) $\lim_{m\rightarrow\infty} \lambda_{m\varphi}^{0}(D)=\infty$.\\
(2) $\partial\Omega$ satisfies Property (P).\\
(3) $\bar{\partial}$-Neumann operator $\mathit{N_1}$ is compact on $L^2_{(0, 1)}(\Omega)$.\\
(4) The Hankel operator $H_{\phi}$ is compact on $A^2 (\Omega)$ for all $\phi\in C^{\infty}(\overline{\Omega})$.\\
(5) $\lim_{m\rightarrow\infty} \lambda_{m\varphi}(D)=\infty$.\\
\end{thm}

\begin{proof} By Theorem $\ref{Catlin;thm1}$, $''(2)\Rightarrow(3)''$ is clear. Also by Proposition $\ref{mj;pro3}$, we can get $''(3)\Rightarrow(4)''$. Moreover, $''(5)\Rightarrow(1)''$ follows from Theorem $\ref{Christ-Fu;thm1}$. 

\medskip

$''(1)\Rightarrow(2)''.$ Denote $W'=\{z\in D\mid\Delta\varphi(z)=0\}$. Let $\{K_j\}_{j=1}^{\infty}$ be a increasing compact sets of $D$, such that $D=\cup_{j=1}^{\infty} K_j$. Therefore, $\partial\Omega=\cup_{j=1}^{\infty} \pi^{-1}(K_j) \cup \pi^{-1}(\partial D)$. Here we follow the notation of lemma $\ref{mj;lem1}$, where $\pi$ denotes the projection from $\partial\Omega$ to $D$. Note that $\lim_{z\rightarrow\partial D} \varphi(z)=+\infty$ as $\partial\Omega$ is smooth. Moreover, $\partial D$ has no fine interior point in $\mathbb{C}$, as $\partial D$ is smooth. By proposition $\ref{mj;pro1}$, $\partial D$ satisfies Property $(P)$ in $\mathbb{C}$. Therefore, $\pi^{-1}(\partial D)=\{(z,0,\dots,0)\mid z\in\partial D\}$ satisfies Property $(P)$ in $\mathbb{C}^{n+1}$.

For each $j$, $K_j\cap W'$ is a compact set in $D$. We can find a sequence of open subsets $\{W_k^{j}\}_{k=1}^{\infty}$ of $D$, such that $(K_j\cap W')\subset\subset W_{k+1}^{j}\subset\subset W_{k}^{j} \subset\subset D$ and $(K_j\cap W')=\cap_k \overline{W_k^{j}}$. By the monotonicity of eigenvalues, for any $(m,k)\in\mathbb{N}\times\mathbb{N}$, $\lambda_{m\varphi}^{0}(D)\leq \lambda_{m\varphi}^{0}(W_k^{j})$. Also, for $u\in C_{0}^{\infty}(W_k^{j})$,
\begin{equation*}
(S_{m\varphi}^{0}u,u)=(-\Delta u+m\Delta\varphi,u)\leq(-\Delta u,u)+(u,u) .
\end{equation*}
When $k$ is big enough relative to $m$, so that $|m\Delta\varphi|\leq 1$. Thus $\lambda(W_k^j)\geq\lambda_{m\varphi}^{0}(W_k^j)-1\geq\lambda_{m\varphi}^{0}(D)-1$. By assumption, we can get that $\lambda(W_k^j)\rightarrow\infty$ as $k\rightarrow\infty$. By (4) of proposition $\ref{mj;pro2}$, it concludes that for each $j$, $K_j\cap W'$ satisfies Property (P). By Lemma $\ref{mj;lem1}$ and $\ref{mj;lem2}$, it implies that the subset of weakly pseudoconvex points of $\pi^{-1}(K_j)$ satisfies Property $(P)$, so does $\pi^{-1}(K_j)$.

By taking countable union of closed sets satisfying Property $(P)$, $\partial\Omega=\cup_{j=1}^{\infty} \pi^{-1}(K_j) \cup \pi^{-1}(\partial D)$ satisfies Property $(P)$ (see for example Corollary 4.14 in \cite{Str2}).

\medskip

$''(4)\Rightarrow(5)''$. First, notice that for any nonzero $m$ and fixed $t>0$, magnetic Schr\"{o}dinger operator $S_{m\varphi}=-[(\partial_x+im\varphi_y)^{2}+(\partial_y-im\varphi_x)^{2}]+m\Delta \varphi=S_{m(\varphi+\log(t))}$; nonmagnetic Schr\"{o}dinger operator $S_{m\varphi}^{0}=-\Delta+m\Delta\varphi=S_{m(\varphi+\log(t))}^{0}$. Thus first eigenvalues stay the same: $\lambda_{m\varphi}=\lambda_{m(\varphi+\log(t))}$, $\lambda_{m\varphi}^{0}=\lambda_{m(\varphi+\log(t))}^{0}$. Applying Lemma $\ref{mj;lem4}$, as  the defining function of rescaled domain $\Omega'$ is $|w_1'|^2+\dots+|w_n'|^2=e^{-2(\varphi(z')+\log(t))}$,  it suffices to prove $\lim_{m\rightarrow\infty} \lambda_{m(\varphi+\log(t))}(D)=\infty$ when $H_{\phi'}$ is compact on $A^{2}(\Omega')$ for all $\phi'\in C^{\infty}(\overline{\Omega'})$. For convenience, we still use $\Omega$ and $\varphi$ instead of $\Omega'$ and $\varphi+\log(t)$ respectively, and thus $\varphi\geq 0$ by the rescaling.

Let $\beta\in C_{0}^{\infty}(D)$, $u_m=\beta(z)w_1^{m} \,d\bar{z}$ and $f_{m}(z,w_1,\dots,w_n)=\bar{\partial}^{*}N(u_m)$. There exists $\phi\in C^{\infty}(\overline{D})$ such that $\partial\phi(z)/\partial\bar{z}=\beta(z)$. Componentwisely, $\partial f_m/\partial\bar{w}_k=0$ for all $1\leq k \leq n$ and $\partial f_m/\partial\bar{z}=\beta(z)w_1^{m}$. Therefore $f_m$ can be written as the following Taylor series: $$f_m (z, w_1,\dots,w_n)=\sum_I f_{mI}(z) w^{I},$$ where $I$ are multi-indexes, since $f_m$ is holomorphic with respect to each $w_k$, $1\leq k \leq n$. By Fubini's Theorem, it is well known that for any ball $B^n(r)$ centered at 0 with radius $r$, $$\int_{B^n (r)} w^{I}\overline{w^{J}}\,d\,V(w_1,\dots,w_n)= 0,$$ when $I\neq J$. It follows that for a given $z$,
\begin{equation*}
\int_{\sum_{k=1}^{n} |w_k|^2<e^{-2\varphi(z)}}w^{I}\overline{w^{J}}\,d\,V(w_1,\dots,w_n)\neq 0 ,
\end{equation*}
unless $I=J$.
So $\bar{\partial} f_m=\sum_{I} \frac{\partial f_{mI}}{\partial\bar{z}} w^{I}\,d\bar{z}$. Also since $f_m$ is orthogonal to $A^{2}(\Omega)$,  $f_{mI}=0$ unless $I=(m,0,\dots,0)$. Thus $f_{m}(z,w_1,\dots,w_n)=\phi(z)w_1^{m}$ where $\partial \phi(z)/\partial\bar{z}=\beta(z)$. By Lemma $\ref{mj;lem3}$,
\begin{equation} \label{mj;eq3.1}
\|H_{\phi}w_1^{m}\|^2=\|\phi(z) w_1^{m}\|^2\leq\epsilon\|\beta(z) w_1^{m}\|\|w_1^{m}\|+C_{\epsilon} \|\beta(z) w_1^m\|_{-1}\|w_1^{m}\| .
\end{equation}
By Lemma $\ref{mj;lem6}$, $\|\beta(z) w_1^{m}\|_{-1}\leq C\|d_{\Omega}(z,w_1,\dots,w_n)\beta(z) w_1^{m}\|$. Then

\begin{equation} \label{mj;eq3.2}
\begin{split}
\|\beta(z) w_1^{m}\|_{-1}^{2}& \leq C^{2}\int_{\Omega} (d_{\Omega}(z,w_1,\dots,w_n))^2 |\beta(z)|^{2}|w_1|^{2m}\,d\,V(z,w_1,\dots,w_n)\\
& \leq C^{2}\int_{D} |\beta(z)|^{2}\,dV(z)\int_{\sum_{k=2}^{n} |w_k|^2<e^{-2\varphi(z)}}\,dV(w_2,\dots,w_n)\\
&\int_{B(0,(e^{-2\varphi(z)}-\sum_{k=2}^{n} |w_k|^2)^{\frac{1}{2}})} (d_{B} (w_1))^{2} |w_1|^{2m}\,dV(w_1)\\
&\leq\frac{C'}{m^2}\|\beta(z){w_1}^{m}\|^{2} .
\end{split}
\end{equation}
\sloppy Here $d_{B} (w_1)$ denotes the distance between $w_1$ and the boundary of ball $B(0,(e^{-2\varphi(z)}-\sum_{k=2}^{n} |w_k|^2)^{\frac{1}{2}})$. Since the distance from a given point to the boundary of $\partial\Omega$ is no more than the distance from that point to $\partial\Omega$ through a given direction, the second inequality follows. The last inequality comes from Lemma $\ref{mj;lem5}$. Therefore combining ($\ref{mj;eq3.1}$) and ($\ref{mj;eq3.2}$),
\begin{equation}\label{mj;eq3.3}
\|\phi(z) w_1^{m}\|^2\leq 2\epsilon\|\beta(z) w_1^{m}\|\|w_1^{m}\| ,
\end{equation}
when $m$ is sufficiently large. Therefore,
\begin{equation*}
\begin{split}
\|\phi(z) w_1^m\|^2 & =\int_{\Omega} |\phi (z)|^{2}|w_1|^{2m}\,d\,V(z,w_1,\dots,w_n)\\
&=\int_{D} |\phi (z)|^{2}\,d\,V(z)\int_{B(0,e^{-\varphi(z)})}  |w_1|^{2m}\,d\,V(w_1)\int_{B(0,(e^{-2\varphi(z)}-|w_1|^2)^{\frac{1}{2}})}\,d\,V(w_2,\dots,w_n)\\
&=\frac{\pi^{n-1}}{(n-1)!}\int_{D} |\phi (z)|^{2}\,d\,V(z)\int_{B(0,e^{-\varphi(z)})}  |w_1|^{2m} (e^{-2\varphi(z)}-|w_1|^2)^{n-1} d\,V(w_1)\\
&=\frac{2\pi^{n}}{(n-1)!}\int_{D} |\phi (z)|^{2} \,d\,V(z)\int_{0}^{e^{-\varphi(z)}}  r^{2m+1} (e^{-2\varphi(z)}-r^2)^{n-1} d\,r\\
&=\frac{2\pi^{n}}{(n-1)!} \sum_{k=0}^{n-1}\frac{(-1)^{k} C_{n-1}^{k}}{2k+2m+2} \int_{D} |\phi (z)|^{2} e^{-2(n+m)\varphi(z)}\,d\,V(z) .
\end{split}
\end{equation*}
By similarly computation for $\|\beta(z){w_1}^{m}\|\|w_1^{m}\|$ and ($\ref{mj;eq3.3}$), we get
\begin{equation*}
\begin{split}
\int_{D} |\phi (z)|^2 e^{-2(n+m)\varphi(z)}\,d\,V(z)&\leq 2\epsilon(\int_{D} |\beta (z)|^{2} e^{-2(n+m)\varphi(z)}\,d\,V(z))^{1/2}(\int_{D} e^{-2(n+m)\varphi(z)}\,d\,V(z))^{1/2}\\
&\leq 2\epsilon'(\int_{D} |\beta (z)|^{2} e^{-2(n+m)\varphi(z)}\,d\,V(z))^{1/2} .
\end{split}
\end{equation*} 
As $\varphi$ is nonnegative by rescaling, the second inequality follows. $\epsilon'$ is a constant depending on $D$ and $\epsilon$.
Therefore, for any $u\in C_{0}^{\infty}(D)$
\begin{equation}\label{mj;eq3.4}
\begin{split}
\int_{D} |u (z)|^2 e^{2(n+m)\varphi(z)}\,d\,V(z)&=\sup\{|<u,\beta>|^2;\beta\in C_{0}^{\infty}(D),\int_{D} |\beta (z)|^2 e^{-2(n+m)\varphi(z)}\,d\,V(z)\leq 1 \}\\
&\leq\sup\{|<u_z,\phi>|^2; \int_{D} |\phi (z)|^2 e^{-2(n+m)\varphi(z)}\,d\,V(z)\leq 2\epsilon' \}\\
&\leq 2\epsilon' \int_{D} |u_z |^2 e^{2(n+m)\varphi(z)}\,d\,V(z) .
\end{split}
\end{equation}
By definition $\ref{mj;eq1.1}$ and (\ref{mj;eq3.4}), $\lambda_{(m+n)\varphi}(D)\rightarrow\infty$ as $m\rightarrow\infty$. Since $n$ is fixed, we obtain $\lambda_{m\varphi}(D)\rightarrow\infty$ as $m\rightarrow\infty$.

\end{proof}

\section*{Acknowledgement}
The author would like to express his gratitude to his advisor, Professor Yuan Yuan, for his guidance and encouragement. We thank Professor Siqi Fu, Professor S\"{o}nmez \c{S}ahuto\u{g}lu and Professor Yunus Zeytuncu for their useful comments. The author is supported in part by National Science Foundation grant DMS-1412384.

\noindent Muzhi Jin, mujin@syr.edu, \\
Department of Mathematics, Syracuse University, Syracuse, NY 13205, USA.

\end{document}